\documentclass[11pt, oneside]{amsart}   	
\usepackage{amsmath,amsthm,amscd,amsfonts, amssymb, mathrsfs}

\usepackage{geometry}                		
\usepackage{graphicx}				

\makeatletter
\@namedef{subjclassname@2020}{\textup{2020} Mathematics Subject Classification}
\makeatother

\newcommand{\sect}[1]{\section{#1}\setcounter{equation}{0}}
\newcommand{\subsect}[1]{\subsection{#1}}

\font\mbn=msbm10 scaled \magstep1
\font\mbs=msbm7 scaled \magstep1
\font\mbss=msbm5 scaled \magstep1
\newfam\mbff
\textfont\mbff=\mbn
\scriptfont\mbff=\mbs
\scriptscriptfont\mbff=\mbss   

\newcommand{\N}       { \mathbb{N}}
\newcommand{\Z}        {\mathbb{Z}  }  
\newcommand\Co           {{\mathbb C}}

\newtheorem{Th}{Theorem}[section]
\newtheorem{Lm}[Th]{Lemma}

\newtheorem{Prop}[Th]{Proposition}

\newtheorem{E}[Th]{Example}

\newtheorem*{Lemma A}{Lemma A}
\newtheorem*{Lemma B}{Lemma B}
\newtheorem*{Lemma C}{Lemma C}

\newtheorem*{Th A}{Theorem A}
\newtheorem*{Th B}{Theorem B}
\begin{document}

\title[Projective Freeness and Stable Rank of Algebras of BV Functions]{Projective Freeness and Stable Rank of Algebras of Complex-valued BV Functions}
\author{Alexander Brudnyi}
\address{Department of Mathematics and Statistics\newline
\hspace*{1em} University of Calgary\newline
\hspace*{1em} Calgary, Alberta, Canada\newline
\hspace*{1em} T2N 1N4}
\email{abrudnyi@ucalgary.ca}


\thanks{Research is supported in part by NSERC}

\begin{abstract} 
The paper investigates the algebraic properties of 
Banach  algebras of complex-valued functions of bounded variation on a finite interval. It is proved that such algebras have Bass stable rank one
 and are projective free if they do not contain nontrivial idempotents.  
 These properties are derived from a new result on the vanishing of the second \v{C}ech cohomology group of the polynomially convex hull of a continuum of a finite linear measure.
 \end{abstract}

\date{}

\maketitle

\sect{Formulation of Main Results}
\subsect{} Let $BV(I)$ be the space of complex-valued functions of bounded variation on the interval $I=[a,b]$. By definition, $f\in BV(I)$ if and only if
\begin{equation}\label{bv}
V_I(f):=\sup \sum_{i=0}^m|f(x_{i+1})-f(x_i)|<\infty,
\end{equation}
where the supremum is taken over all partitions $a=x_0<x_1<\cdots <x_m=b$, $m\in\N$, of $I$. In this paper, we study algebraic properties of 
 Banach algebras  of $BV(I)$ functions.
  
 Recall that a unital commutative ring $R$  
is said to be {\em projective free} if every finitely generated projective $R$-module is free 
(i.e., if $M$ is an  $R$-module such that $M\oplus N\cong R^n$  for an $R$-module $N$ and  $n \in \Z_+\, (:=\N\cup\{0\})$, then $M\cong R^m$ for some  $m \in \Z_+$).  Let $M_n(R)$ denote the ring of 
$n\times n$ matrices over $R$ and $GL_n(R)$ its unit group.  In terms of matrices, the ring $R$ is projective free if and only if 
 for each $n\in \N$ every  $X\in M_n(R)\setminus  \{0_n, I_n\}$ such that $X^2=X$ (i.e., an idempotent)
has a form $X=S(I_r\oplus  0_{n-r})S^{-1}$  for some $ 
S\in GL_n(R)$, $r\in \{1,\dots, n-1\}$; here  $0_k$ and $I_k$ are zero  and  identity matrices in $M_k(R)$; see \cite[Prop.\,2.6]{C}. (For some examples of projective free rings and their applications, see, e.g., \cite{BS}, \cite{L}, \cite{Vi} and references therein.)

Let $A\subset \ell_\infty(I)$ be a complex Banach function algebra such that
the subalgebra $A\cap BV(I)$ is dense in $A$. We denote by $1_I$ the unit of $A$ (i.e., the constant function of value $1$ on $I$). For a nonzero idempotent $p\in A$, we set $A_p:=\{pg\, :\, g\in A\}$. Then $A_p$ is a closed subalgebra of $A$ with unit $p$.   If $M$ is an $A$-module, then $M_p:=\{pm\, :\, m\in M\}$ is a submodule which can be regarded as an
$A_p$-module as its annihilator contains ${\rm ker}\,(1_I-p)$. 

\begin{Th}\label{te1.1}
Let $M$ be a finitely generated projective $A$-module. Then there exists idempotents $p_1,\dots, p_k\in A$ such that $M=\oplus_{i=1}^k\, M_{p_i}$ and each $M_{p_i}$ is a free $A_{p_i}$-module. In particular, if $A\subset C(I)$, then it is a projective free ring.
\end{Th}
For instance, Theorem \ref{te1.1} holds true for complex Banach  function algebras $A\subset BV(I)$ and their uniform closures $\bar A\subset\ell_\infty(I)$. \smallskip

Let $A$ be an associative ring with unit. For a natural number $n$, let
$U_n(A)$ denote the set of {\em unimodular} elements of $A^n$, i.e.,
\[
U_n(A)=\left\{(a_1,\dots, a_n)\in A^n\, :\,  Aa_1+\cdots +Aa_n=A\right\}.
\]
An element $(a_1,\dots, a_n)\in U_n(A)$ is called {\em reducible} if there exist $c_1,\dots, c_{n-1}\in A$ such that
$
(a_1+c_1 a_n,\dots, a_{n-1}+c_{n-1}a_n)\in U_{n-1}(A).
$
The {\em stable rank} of $A$ is the least $n$ such that every element of $U_{n+1}(A)$ is reducible. 
The concept  of the stable rank introduced by Bass \cite{B} plays an important role in some stabilization problems of algebraic $K$-theory.  
Following Vaserstein \cite{V} we call a ring of stable rank $1$ a $B$-ring. (We refer to this paper for some examples and properties of $B$-rings.) 
\begin{Th}\label{te1.2}
Each  complex Banach  function algebra $A\subset BV(I)$ is a $B$-ring.
\end{Th}
\begin{E}\label{ex1.3}
{\rm Theorems \ref{te1.1} and \ref{te1.2} are applicable to  closed subalgebras of the following  function algebras:
(a) $(BV(I),\lVert\cdot\rVert_{BV})$, where $\lVert f\rVert_{BV}:=\sup_{I}|f|+V_I(f)$;
(b) $(AC(I),\lVert\cdot\rVert_{AC})$ - the algebra of absolutely continuous complex-valued functions on $I$, where $\|f\|_{AC}:=\max_I |f|+\int_I\, |f'(t)|\,dt$; (c)
$({\rm Lip}(I),\lVert\cdot\rVert_{{\rm Lip}})$ - the algebra of complex-valued Lipschitz functions on $I$, where $\|f\|_{\rm Lip}:=\max_I |f|+\sup_{x\ne y}\frac{|f(x)-f(y)|}{|x-y|}$;
(d) $(C^k(I),\lVert\cdot\rVert_{C^k})$ - the algebra of complex-valued $C^k$ functions on $I$, $k\ge 1$, where $\|f\|_{C^k}:=\sum_{i=0}^k \max_I |f^{(i)}|$.
}
\end{E}
\subsect{}
Theorems \ref{te1.1} and \ref{te1.2} are derived from  a general result presented in this section. For its formulation, recall that for a commutative unital complex Banach algebra $A$,  the maximal ideal space $\mathfrak M(A)\subset A^\ast$  
 is the set of nonzero homomorphisms $A \!\rightarrow\! \Co$ endowed with the {\em Gelfand topology}, the weak-$\ast$ topology of  $A^\ast$. It is a compact Hausdorff space contained in the unit sphere of $A^\ast$. The {\em Gelfand transform} defined by $\hat{a}(\varphi):=\varphi(a)$ for $a\in A$ and $\varphi \in \mathfrak M(A)$ is a nonincreasing-norm morphism from $A$ into  $C(\mathfrak M (A))$, the Banach algebra of complex-valued continuous functions on $\mathfrak M(A)$. Also,  recall that the {\em covering dimension} of a topological space $X$, denoted by ${\rm dim}\, X$, is the smallest integer $d$ such that every open cover of $X$ has an open refinement of order at most $d+1$. If no such integer exists, then $X$ is said to have infinite covering dimension.

 \begin{Th}\label{te1.5}
 Let $A\subset \ell_\infty(I)$ be a complex Banach  function algebra such that $A\cap BV(I)$ is dense in $A$. Then
 ${\rm dim}\,\mathfrak M(A)\le 2$ and the \v{C}ech cohomology group $H^2(\mathfrak M(A),\Z)=0$.
 \end{Th}
 Note that for a complex Banach  function algebra $A\subset \ell_\infty(I)$ with uniform closure $\bar A$ 
 the maximal ideal spaces  $\mathfrak M(A)$ and $\mathfrak M(\bar A)$ are homeomorphic, see, e.g., \cite[Prop.\,3]{R}.
  \begin{E}\label{ex1.6}
 {\rm (1) As a Banach algebra, $BV(I)=\ell_1(I)\rtimes BV_+(I)$ - the semidirect  product of the closed ideal $\ell_1(I)$ and the Banach subalgebra $BV_+(I)$ of right-continuous $BV$ functions, see, e.g., \cite[Cor.\,2.2]{BB}.  Thus, the uniform closure $\overline{BV(I)}=c_0(I)\rtimes R_+(I)$ - the semidirect product 
 of the closed ideal $c_0(I)$ of functions with at most countable supports converging to $0$ and the Banach subalgebra $R_+(I)\subset\ell_\infty(I)$
of right-continuous functions having first kind discontinuities. Then each homomorphism in $\mathfrak M(\overline{BV(I)})$ is uniquely determined by its restrictions to $c_0(I)$ and $R_+(I)$. This leads to a continuous injection
$r:\mathfrak M(\overline{BV(I)})\to \mathfrak M(c(I))\times\mathfrak M(R_+(I))$, where $c(I):=\Co\cdot 1_I\oplus c_0(I)$. Next, $\mathfrak M(c(I))$ is homeomorphic to the one-pointed compactification of the discrete set $I$. In particular, ${\rm dim}\,\mathfrak M(c(I))=0$. In turn, there is a continuous surjection $p: \mathfrak M(R_+(I))\to \mathfrak M(C(I))=I$, the transpose of the embedding $C(I)\hookrightarrow R_+(I)$, whose fibres consist of two points over interior points of $I$ and of one point over the endpoints of $I$. (Specifically, if $\varphi\in p^{-1}(x)$, then $\varphi(f)$ is equal either to $f(x^-)$ or to $f(x^+)$.)
Moreover, ${\rm dim}\,\mathfrak M(R_+(I))=0$, see, e.g., \cite[Thm.\,1.7]{BK}.
These imply that $r$ embeds $\mathfrak M(\overline{BV(I)})$ into the zero-dimensional compact Hausdorff space $\mathfrak M(c(I))\times\mathfrak M(R_+(I))$; hence, ${\rm dim}\,\mathfrak M(BV(I))={\rm dim}\,\mathfrak M(\overline{BV(I)})=0$  and $H^i(\mathfrak M(BV(I)),\Z)=0$ for all $i\in\N$.\smallskip

 \noindent (2) If $A$ is one of the algebras of (b),(c) or (d) of Example \ref{ex1.3}, then $\mathfrak M(\bar A)$ is homeomorphic to $I$ and, hence, ${\rm dim}\, \mathfrak M(\bar A)=1$  and $H^i(\mathfrak M(\bar A),\Z)=0$ for all $i\in\N$.\smallskip
 
 \noindent (3) If $A$ is generated by $f_1,\dots, f_n\in BV(I)$, then $\mathfrak M(A)$ is homeomorphic to the polynomially convex hull of the  range of $(f_1,\dots, f_n): I\to\Co^n$ described by the Alexander theorem \cite{A} presented in the next section.
 }
 \end{E}
 \subsect{}
In the sequel, $\mathscr H^1$ denotes Hausdorff $1$-dimensional measure.
Also,
 \[
\widehat{K}:=\bigl\{z\in\Co^n\, :\, |p(z)|\le \mbox{$\sup_{K}$}\,|p|\quad \forall p\in\Co [z_1,\dots, z_n]\bigr\}
\]
stands for the polynomially convex hull of a bounded subset $K\subset\Co^n$.
If $X\Subset\Co^n$, then $P(X)\subset C(X)$ denotes the uniform closure of the restriction of polynomials $\Co [z_1,\dots, z_n]|_X$. A compact connected subset of $\Co^n$ is called a {\em continuum}.

 The following result is due to Alexander \cite[Thm.\,1]{A}.
 \begin{Th A}
Suppose $\Gamma\subset\Co^n$ is a compact subset of a continuum of finite $\mathscr H^1$-measure. Then $\widehat{\Gamma}\setminus\Gamma$ is a (possibly empty) pure $1$-dimensional complex analytic subset of $\Co^n\setminus\Gamma$. If $H^1(\Gamma,\Z)=0$, then $\widehat\Gamma=\Gamma$ and $P(\Gamma)=C(\Gamma)$.
 \end{Th A}
 For historical remarks and further developments related to this theorem, see  \cite{S}.
 
Using Theorem A we prove the following:
\begin{Th}\label{te1.4}
Suppose $\Gamma\subset\Co^n$ is a compact subset of a continuum of finite $\mathscr H^1$-measure. Then (a) ${\rm dim}\,\widehat\Gamma\le 2$; (b)
$H^2(\widehat\Gamma,\Z)=0$.
\end{Th}
Theorem \ref{te1.5} is derived from Theorem \ref{te1.4}.

\sect{Proof of Theorem \ref{te1.4}}
(a) Let ${\rm dim}_{\mathscr H}$ denote the Hausdorff dimension.
By the Szpilrajn theorem, see, e.g., \cite[pp.\,62-63]{H}, and because $\mathscr H^1(\Gamma)<\infty$, 
\begin{equation}\label{e1.3.1}
 {\rm dim}\,\Gamma\le {\rm dim}_{\mathscr H}\Gamma=1.
\end{equation}
In turn, as  $\widehat\Gamma\setminus\Gamma\ne\emptyset$ is a $1$-dimensional complex analytic space, its compact subsets have covering dimension $\le 2$. These  imply that ${\rm dim}\,\widehat\Gamma\le 2$,
 see, e.g., \cite[Ch.\,2,\,Th.\,9-11]{N}.\smallskip
 
 (b) Let $E\subset\Co^n$ be a continuum with $\mathscr H^1(E)<\infty$ containing $\Gamma$. Then according to \cite[Ch.\,3,\,Exercise~3.5]{F}, there are  functions $f_1,\dots, f_{n}\in {\rm Lip}(I)$ such that
\[
E\subset K:=(f_1,\dots, f_{n})(I)\quad {\rm and}\quad \mathscr H^1(K)\le 2\mathscr H^1(E).
\]  
By $A\subset {\rm Lip}(I)$ we denote the unital complex closed subalgebra generated by $f_1,\dots, f_{n}$.
The maximal ideal space of $A$ is naturally identified with $\widehat K$, see, e.g., \cite[Ch.III,\,Thm.\,1.4]{G}. According to part (a) of the theorem,  ${\rm dim}\,\widehat K\le 2$. We use the following 
\begin{Prop}\label{prop2.1}
Suppose $\mathcal A$ is a complex Banach  function algebra defined on its maximal ideal space $\mathfrak M(\mathcal A)$. If ${\rm dim}\,\mathfrak M(\mathcal A)\le 2$, then there are bijections
\begin{itemize}
\item[(a)]
$c_1:{\rm Vect}_1(\mathfrak M(\mathcal A))\to H^2(\mathfrak M(\mathcal A),\Z)$,
where ${\rm Vect}_1(\mathfrak M(\mathcal A))$ is the set of isomorphism classes of complex rank one vector bundles over $\mathfrak M(\mathcal A)$;
\item[(b)]
$h: [\mathfrak M(\mathcal A),\mathbb S^2]\to  H^2(\mathfrak M(\mathcal A),\Z)$, 
where $[\mathfrak M(\mathcal A),\mathbb S^2]$ is the set of homotopy classes of continuous maps from $\mathfrak M(\mathcal A)$ to the two-dimensional unit sphere $\mathbb S^2$; 
\item[(c)]
$i:[{\rm ID}_1(\mathcal A_2)]\to [\mathfrak M(\mathcal A),\mathbb S^2]$, 
where $[{\rm ID}_1(\mathcal A_2)]$ is the set of connectivity components of the class of idempotent $2\times 2$ matrices with entries in $\mathcal A$ of constant rank $1$.
\end{itemize}
\end{Prop}
\begin{proof}
In (a) and (c) condition ${\rm dim}\,\mathfrak M(\mathcal A)\le 2$ is not required. In fact,
bijection $c_1$ is determined by assigning to a bundle its first Chern class
while existence of bijection $i$ follows from the Novodvorskii-Taylor theory, see     \cite[\S5.3,\,p.\,186]{T}. Finally, existence of bijection $h$ under condition ${\rm dim}\,\mathfrak M(\mathcal A)\le 2$ follows from
the Hopf theorem, see, e.g., \cite{Hu}.
\end{proof}
Let us proceed with the proof of the theorem. We set  $F:=(f_1,\dots, f_n): I\to\Co^n$.  The algebra $A$ is isomorphic to its Gelfand transform $\hat{A}$ - a compex Banach function algebra on $\widehat K=\mathfrak M(A)$ with norm induced from $A$ such that  $F^*\hat A=A$. Since $\mathfrak M(\hat A)=\widehat K$ as well,
 in the notation of the proposition, each 
 \begin{equation}\label{eq2.1}
G=\left[
\begin{array}{cc}
g_1&g_2\\
g_3&g_4
\end{array}
\right]\in {\rm ID}_1(\hat A_2)
\end{equation}
can be viewed as a map from $\widehat K$ to $M_2(\Co)$ with coordinates in the algebra $\hat A$ whose image ${\rm ID}_1(\Co_2)$ consists of idempotent matrices of rank $1$. Thus,
\[
Z=\left[
\begin{array}{cc}
z_1&z_2\\
z_3&z_4
\end{array}
\right]
\in {\rm ID}_1(\Co_2)
\]
if and only if ${\rm rank}\, Z=1$ and
\[
Z^2-Z=\left[
\begin{array}{ccc}
z_1&z_2\\
z_3&z_4
\end{array}
\right]\cdot
\left[
\begin{array}{ccc}
z_1-1&z_2\\
z_3&z_4-1
\end{array}
\right]=
\left[
\begin{array}{ccc}
0&0\\
0&0
\end{array}
\right] 
\]
which implies
\begin{equation}\label{eq2.2}
z_4=1-z_1\quad {\rm and}\quad 
\begin{array}{ccc}
\displaystyle z_3=\frac{z_1(z_1-1)}{z_2}&{\rm if}&z_2\ne 0,\medskip\\
z_1\in\{0,1\},\ z_3\in\Co&{\rm if}&z_2=0.
\end{array}
\end{equation}

Let $S:=G(K)=(G\circ F)(I)\subset M_2(\Co)\cong\Co^4$. Since the entries of the map $G\circ F$ lie in ${\rm Lip}(I)$,
$S$ is a continuum with $\mathscr H^1(S)<\infty$.
Moreover, $G(\widehat K)\subset\widehat S\subset  {\rm ID}_1(\Co_2)$ and ${\rm dim}\,\widehat S\le 2$ by part (a) of the theorem. By the definition, the identity map
$\Co^4 \supset {\rm ID}_1(\Co_2)\to  {\rm ID}_1(\Co_2)\subset M_2(\Co)$ determines the holomorphic idempotent $Z$ on ${\rm ID}_1(\Co_2)$ whose pullback by $G$ coincides with $G\in {\rm ID}_1(\hat A_2)$. According
to Proposition \ref{prop2.1}, $Z|_{\widehat S}$ determines a complex rank one vector bundle over $\widehat S$ whose triviality implies that $Z|_{\widehat S}$ and, hence, $G$ belong to the connectivity components (in
${\rm ID}_1(C(\widehat S)_2)$ and ${\rm ID}_1(\hat A_2)$, respectively) of the constant idempotent $I_1\oplus 0_1$. If the latter is true for all $G\in {\rm ID}_1(\hat A_2)$, then Proposition \ref{prop2.1} implies that
$H^2(\widehat K,\Z)=0$. But  $\widehat{\Gamma}\subset\widehat K$ and ${\rm dim}\,\widehat K\le 2$ and so the previous condition implies that $H^2(\widehat \Gamma,\Z)=0$ by the Hopf theorem, as required. Thus, to complete the proof of the theorem,  it suffices to prove
\begin{Lm}\label{lem2.2}
Each complex rank one vector bundle over $\widehat S$ is trivial.
\end{Lm}
\begin{proof}
Let $V$ be a complex rank one vector bundle over $\widehat S$. Consider the linear projection $\pi: \Co^4\to\Co^2$, $\pi(z_1,\dots, z_4):=(z_1,z_2)$.  For
$w_i=(i,0)$, $i=0,1$, sets $ \pi^{-1}(w_i)\cap {\rm ID}_1(\Co_2)$ are biholomorphic to $\Co$, see \eqref{eq2.2}. Thus, $Z_i:=\pi^{-1}(w_i)\cap\widehat S$ are homeomorphic to compact subsets of $\Co$; hence, $H^2(Z_i,\Z)=0$, $i=0,1$. This implies that $V|_{Z_i}$ are trivial bundles. In particular, there are disjoint open neighbourhoods $U_i\subset\widehat S$ of
$Z_i$ and 
nonvanishing continuous sections $s_i:U_i\to V$, $i=0,1$. Since
\[
Z_i=\bigcap_{O_i\in \mathcal N(w_i)}\pi^{-1}(O_i)\cap\widehat S,\quad i=0,1,
\]
where $\mathcal N(w_i)$ is the set of all open neighbourhoods of $w_i$, without loss of generality we may assume that $U_i=\pi^{-1}(O_i)\cap\widehat S$ for some $O_i\in \mathcal N(w_i)$, $i=0,1$. Let  $U_2,\dots, U_k$ be relatively compact open subsets of  $\widehat S\setminus \pi^{-1}(w_i)$ such that $(U_i)_{i=0}^k$ is an open cover of $\widehat S$ and each $V|_{U_i}$ is trivial. 
Due to \eqref{eq2.2}, $\pi$ maps $\widehat S\setminus (\pi^{-1}(w_0)\cup\pi^{-1}(w_1))$ homeomorphically onto $\pi(\widehat S)\setminus \{w_0,w_1\}$. Thus, there exist (relatively compact) open subsets $O_i\subset
\pi(\widehat S)\setminus \{w_0,w_1\}$, $2\le i\le k$, such that $\pi^{-1}(O_i)\cap\widehat S=U_i$. Let $s_i: U_i\to V$ be nonvanishing continuous sections of $V|_{U_i}$, $i=2,\dots, k$. Then $V$ is determined by a continuous cocycle $\{c_{ij}\}_{0\le i,j\le k}$,
\[
c_{ij}:=s_i^{-1}\cdot s_j\in C(U_i\cap U_j,\Co^*), \quad 0\le i,j\le k.
\]
Further, since each nonvoid $U_i\cap U_j$ is a subset of $\widehat S\setminus (\pi^{-1}(w_0)\cup\pi^{-1}(w_1))$, due to \eqref{eq2.2} there exist $d_{ij}\in C(O_i\cap O_j,\Co^*)$ such that $\pi^{*}d_{ij}=c_{ij}$.  The family $\{d_{ij}\}_{0\le i,j\le k}$ is a $1$ cocycle on the cover $(O_i)_{0\le i\le k}$ of $\pi(\widehat S)$ which determines a bundle $V'$ on $\widehat S$ such that
$\pi^*V'|_{\widehat S}=V$. Thus to complete the proof it suffices to show that $V'$ is a trivial bundle. 

Indeed, by definition, $\pi(\widehat S)\subset \widehat{\pi(S)}\subset\Co^2$. Since $\pi(S)=:(\pi\circ G\circ F)(I)$ is a continuum with $\mathscr H^1(\pi(S))<\infty$, part (a) of the theorem implies that ${\rm dim}\,  \widehat{\pi(S)}\le 2$. In addition, $\widehat{\pi(S)}$ is a polynomially convex subset of $\Co^2$; hence, $H^2(\widehat{\pi(S)},\Z)=0$ (see, e.g., \cite[Cor.\,2.3.6]{S}). These imply that
$H^2(\pi(\widehat S),\Z)=0$ by the Hopf theorem. In particular, each complex rank one vector bundle over $\pi(\widehat S)$ is trivial; hence the bundle $V'$ is trivial as well, as required.
\end{proof}
The proof of Theorem \ref{te1.4} is complete.
\sect{Proofs of Theorems \ref{te1.1}, \ref{te1.2}, \ref{te1.5}}
\begin{proof}[Proof of Theorem \ref{te1.5}]
Since $A\cap BV(I)$ is dense in $A$,
the algebra $A=\varinjlim A_\alpha$ - the injective limit of the injective system $\{A_\alpha, i_{\alpha}^\beta\}$ of
finitely generated subalgebras of $A\cap BV(I)$; here $i_{\alpha}^\beta: A_\alpha\hookrightarrow A_\beta$
is the inclusion. In turn, the maximal ideal space $\mathfrak M(A)=\varprojlim\mathfrak M(A_\alpha)$ - the projective limit of the
adjoint projective system $\{\mathfrak M(A_\alpha), (i_{\alpha}^\beta)^*\}$
of the maximal ideal spaces, see, e.g., \cite[Prop.\,9]{R}.
Suppose  $A_\alpha$ is generated by  $f_{1,\alpha},\dots, f_{k_{\alpha},\alpha}\in BV(I)$. Then the range of the map $F_{\alpha}=(f_{1,\alpha},\dots, f_{k_{\alpha},\alpha}): I\to\Co^{k_\alpha}$ denoted by $\Gamma_\alpha$ is contained in a continuum of finite $\mathscr H^1$ measure, see, e.g., \cite[Ch.\,3,\,Exercise\,3.1]{F}. Moreover, $\mathfrak M(A_\alpha)$ is homeomorphic to $\widehat{\Gamma}_\alpha$, see, e.g., \cite[Ch.\,III,\,Thm.\,1.4]{G}, where ${\rm dim}\,\widehat{\Gamma}_\alpha\le 2$  and $H^2(\widehat{\Gamma}_\alpha,\Z)=0$ by Theorem
\ref{te1.4}. Then, since $\mathfrak M(A)=\varprojlim\mathfrak M(A_\alpha)$, ${\rm dim}\,\mathfrak M(A)\le 2$, see, e.g., \cite[Thm.\,3.3.6]{E}, and
$H^2(\mathfrak M(A),\Z)=0$, see, e.g., \cite[Thm.\,3.1,\,p.\,261]{EM}, as required.
\end{proof}
\begin{proof}[Proof of Theorem \ref{te1.1}]
Let $M$ be a finitely generated projective $A$-module determined by 
an idempotent $I\in M_n(A)$.
The rank of $M$ is a continuous  $\Z_+$-valued function on $\mathfrak M(A)$ equal to the rank of the Gelfand transform of $I$ at points of $\mathfrak M(A)$ (see, e.g., \cite[\S7.6]{T}).  Let $0\le i_1<\cdots < i_k\le n$ be the range of this function and $\mathfrak M_s\subset\mathfrak M(A)$ be the clopen subset where $\hat{I}$ has constant rank $i_s$.
Then $\mathfrak M(A)=\sqcup_{s=1}^k\mathfrak M_s$ and 
by the Shilov idempotent theory, see, e.g., \cite[Ch.\,III,\,Cor.\,6.5]{G},
there exist idempotents $p_1,\dots, p_k\in A$ with $\sum_{s=1}^k p_s= 1_A$ such that the maximal ideal space of $A_{p_s}$ is $\mathfrak M_s$. We have $A=\oplus_{s=1}^k A_{p_s}$ which leads to the decomposition $I=\oplus_{s=1}^k\,p_s\cdot I$, where $p_s\cdot I\in M_n(A_{p_s})$ is the idempotent determining the projective $A_{p_s}$-module $M_{p_s}$.

Next, due to the Novodvorskii-Taylor theory, see \cite[\S7.5,\,Thm.]{T}, there exists a bijection between
 isomorphism classes of finitely generated projective $A_{p_s}$-modules and 
 complex vector bundles over $\mathfrak M_s$. In our case, the isomorphism class of $M_{p_s}$ corresponds to the isomorphism class of a bundle $E_s$ over $\mathfrak M_s$ of complex rank $i_s$.
 Since ${\rm dim}\, \mathfrak M_s\le 2$ and
$H^2(\mathfrak M_s,\Z)=0$ by Theorem \ref{te1.5}, the bundle $E_s$ is trivial 
(i.e., isomorphic to $\mathfrak M_s\times\Co^{i_s}$) which implies that $M_{p_s}$ is isomorphic to the free module $(A_{p_s})^{i_s}$, as required.
\end{proof}
\begin{proof}[Proof of Theorem \ref{te1.2}]
Let $J\subset A\subset BV(I)$ be a closed ideal. Its hull $\mathcal Z( J)\subset\mathfrak M(A)$ is given by 
\[
\mathcal Z(J):=\{x\in\mathfrak M(A)\,:\,\hat f(x)=0\quad\forall f\in J\}.
\]
Consider the closed unital subalgebra $A_J:=\{c\cdot  1_A+f\,:\, c\in\Co,\ f\in J\}\subset A$. By $Q_J: \mathfrak M(A)\rightarrow \mathfrak M(A_J)$ we denote the continuous map transposed to the embedding $A_J\hookrightarrow A$. Then $Q_J$ is a surjection which is one-to-one on $\mathfrak M(A)\setminus\mathcal Z(J)$ and sends $\mathcal Z(J)$ to a point, see, e.g., \cite[Prop.\,2.1]{Br} for the proof of a similar result. On the other hand,
$A_J\subset BV(I)$ and so by Theorem \ref{te1.5}, ${\rm dim}\,\mathfrak M(A_J)\le 2$ and $H^2(\mathfrak M(A_J),\Z)=0$.

According to \cite[Thm.\,1.3]{Su}, to prove that the stable rank of $A$ is $1$ it suffices to show that the relative \v{C}ech cohomology groups $H^2(\mathfrak M(A),\mathcal Z(J),\Z)=0$ for all ideals $J\subset A$.
However, due to the {\em strong excision property} for cohomology, see, e.g., \cite[Ch.\,6,\,Thm.\,5]{Sp}, the pullback map $Q_J^*$ induces an isomorphism of the \v{C}ech cohomology groups $H^2(\mathfrak M(A_J),\Z)\cong H^2(\mathfrak M(A),\mathcal Z(J),\,\Z)$.  In particular, $H^2(\mathfrak M(A),\mathcal Z(J),\,\Z)=0$, as required.
\end{proof}
There are some applications of Theorems \ref{te1.1} and \ref{te1.2} to operator-valued $BV(I)$ functions and to interpolating problems for $BV(I)$ maps into some complex manifolds  analogous to those of \cite[Sec.\,1.2]{Br1}  and \cite[Thms.\,1.4,\,1.6]{Br2}.  These results will be published elsewhere.

 \end{document}